\documentclass[12pt]{amsart}
\usepackage{amsmath}
\usepackage{amsthm}
 \usepackage[english]{babel}
\usepackage{amssymb,verbatim}
\usepackage{mathrsfs}
\usepackage[T1]{fontenc}
\usepackage{graphicx}
 \usepackage{hyperref}
\makeatletter
\@namedef{subjclassname@2010}{%
  \textup{2010} Mathematics Subject Classification}
\makeatother

\newcommand{\PSH}{\mbox{$\mathcal{PSH}$}}

\newcommand{\ddc}{\operatorname{dd^{c}}}
\newcommand{\C}{\mathbb{C}}
\newcommand{\R}{\mathbb{R}}

\newcommand{\abs}[1]{\left|#1\right|}

\newcommand{\Cn}{\mathbb{C}^n}
\newcommand{\diam}{\operatorname{diam}}
\renewcommand{\bar}[1]{\overline{#1}}
\renewcommand{\epsilon}{\varepsilon}
\renewcommand{\tilde}[1]{\widetilde{#1}}
\renewcommand{\hat}[1]{\widehat{#1}}

\newtheorem{theorem}{Theorem}
\newtheorem{corollary}[theorem]{Corollary}

\newtheorem{lemma}[theorem]{Lemma}

\newtheoremstyle{note}
  {3pt}
  {3pt}
  {}
  {}
  {\itshape}
  {:}
  {.5em}
  {}

\newtheorem{note}{Remark}

\theoremstyle{definition}
\newtheorem{definition}{Definition}
\newtheorem{example}{Example}

\begin{document}

\date{\today}

\author[B. Avelin]{Benny Avelin} \address{Department of Mathematics \\ Uppsala University \\ SE-751 06 Uppsala \\ Sweden}\email{benny.avelin@math.uu.se}
\author[L. Hed]{Lisa Hed}\address{Department of Mathematics and Mathematical Statistics\\ Ume\aa \ University\\SE-901 87 Ume\aa \\ Sweden}\email{lisa.hed@math.umu.se}
\author[H. Persson]{H\aa{}kan Persson}\address{Department of Mathematics \\ Uppsala University \\ SE-751 06 Uppsala \\ Sweden}\email{hakan.persson@math.uu.se}

\title[Plurisubharmonic Functions Beyond Lipschitz Domains]{Approximation and Bounded Plurisubharmonic Exhaustion Functions Beyond Lipschitz Domains}
\begin{abstract}
	Using techniques from the analysis of PDEs to study the boundary behaviour of functions on domains with low boundary regularity, we extend results by Forna\ae{}ss-Wiegerinck (1989) on plurisubharmonic approximation and by Demailly (1987) on the existence on bounded plurisubharmonic exhaustion functions to domains beyond Lipschitz boundary regularity. 
\end{abstract}



\maketitle

\section{Introduction}
Many results of complex analysis in several varibles have been stated and proved in the realm of smooth domains, and little interest has been directed towards domains with lower boundary regularity. In the study of PDEs, on the other hand, much effort has been put into extending results from smooth domains, into Lipschitz domains and beyond. The aim this paper is to use some of these techniques to extend two well-known results on plurisubharmonic functions to domains beyond Lipschitz boundaries. Common for both theorems is that it is the notion of boundary regularity, rather than idea of proof that sets the limits for the theorem.

The first theorem concerns the uniform approximation of continuous plurisubharmonic functions on a domain with $C^0$ boundary by continuous plurisubharmonic functions defined in neighborhoods of the closure of the domain. The proof relies on the fact that we can extend plurisubharmonic functions across a continuous boundary. This is possible because there exists a way to translate a piece of the boundary an arbitrary small amount, without it crossing the untranslated part. This property is shared among all domains which are given by graphs, but not by all fractal domains.

 The other main result is concerned with hyperconvexity. We asked the question what, in terms of boundary regularity, do we need to impose on a pseudoconvex domain in addition to be able to conclude the existence of a bounded plurisubharmonic exhaustion function? We extend a previous result by Demailly, \cite{Dem}, from Lipschitz domains to a class of domains that includes all Log-Lipschitz domains, by carefully studying the asymptotics of the proof in \cite{Dem}.

We hope that these results will spark some interest in complex analysis for domains with low regularity.

To properly state our results we need to introduce some notation. Points in Euclidean $n$-space $ \R^{n} $ are denoted by $ x = ( x_1, \dots, x_n) $ or $ ( x', x_n ) $ where $ x' = ( x_1, \dots, x_{n-1} ) \in \R^{n - 1}$, and and we denote by $\{e_j\}_{j=1}^n$ the standard basis of $\mathbb{R}^n$. Let $ \bar E, \partial E, \diam E $ be the closure, boundary, and diameter of $E.$ Let $ | x | = (x \cdot x )^{1/2}, $ be the Euclidean norm of. Given $ x \in \R^{n}$ and $r >0$, let $ B (x, r ) = \{ y \in \R^{n} : | x - y | < r \}$, and let $B'(x,r) = \{ y' \in \R^{n-1} : | x' - y' | < r \}$. Given $ E, F \subset \R^n, $ let $ \delta ( E, F ) $ be the Euclidean distance from $ E $ to $ F $. In case $ E = \{y\}, $ we write $ \delta ( y, F ) $ and if it is obvious which set $F$ we mean, then we simply write $\delta(y)$.
If $ O \subset \R^{n} $ is open and $ k = 0,1,\ldots$ and $\gamma \in [0,1]$ then by $ C^{k ,\gamma} ( O ), $ we denote the space of $k$ times continuously differentiable functions, such that all partial derivatives up to order $k$ are H\"older continuous with exponent $\gamma$, $\gamma = 0$ means that we just have a continuous function, $\gamma = 1$ is usually called Lipschitz continuous. All these definitions makes sense in $\C^n$ with the canonical identification $\C^n \simeq \R^{2n}$.

Let $\PSH(\Omega)$ denote the class of pluri\-sub\-harm\-o\-n\-ic functions on a bounded domain $\Omega$ in $\C^n$, and $\PSH(\overline \Omega)$ the class of function plurisubharmonic on a neighbourhood of $\Omega$. 
Let $\ddc = 2i \partial \bar \partial$, if $u \in \PSH$, we will take $\ddc u$ in a distributional sense.
We say that a domain $\Omega \subset \C^n$ is pseudoconvex, if there exists $\phi \in \PSH(\Omega) \cap C(\Omega)$ such that $\phi(z) \to + \infty$ as $z \to \partial \Omega$, we call $\phi$ a \emph{plurisubharmonic exhaustion function}.
We say that $\phi \in \PSH(\Omega) \cap C(\bar \Omega)$ is a \emph{bounded plurisubharmonic exhaustion function} if  $\phi < 0$ in $\Omega$ and $\phi = 0$ on $\partial \Omega$. If there exists such an exhaustion function, $\Omega$ is often said to be hyperconvex. Note that all hyperconvex domains are pseudoconvex.

\section{Main Results} 
\label{sec:main_results}
In this section we state in detail the main results of the paper. The exact definitions of the the different boundary regularities in the statements of the theorems are postponed to Section \ref{sec:domains_with_low_regularity_rewrite_}.
\subsection{The PSH-Mergelyan property in $C^0$ domains}
In 1987, Sibony \cite[Theorem 2.2]{sibony} proved that if $\Omega$ is a $C^\infty$ smooth pseudoconvex domain, it is possible to approximate functions in $\PSH(\Omega) \cap C(\overline \Omega)$ with functions in $\PSH(\overline \Omega) \cap C^\infty(\overline \Omega)$. This approximation property corresponds to the well-known Mergelyan property, but for plurisubharmonic functions instead of holomorphic functions. In the same paper Sibony also asked the question whether his result could be extended to pseudoconvex $C^1$ domains.

In 1989, Forn\ae ss and Wiegerinck \cite[Theorem 1]{for_wieg} proved that this is indeed the case, and that the domain $\Omega$, does not even need to be pseudoconvex.

In this paper we prove the following even stronger result.
\begin{theorem} \label{theorem:c0bdry}
Let $\Omega \subset \C^n$ be a bounded domain with $C^0$-boundary. Then every function in $\PSH(\Omega) \cap C(\overline \Omega)$ can be uniformly approximated on $\bar\Omega$ by functions in $\PSH(\overline \Omega) \cap C^\infty(\overline \Omega)$.
\end{theorem}
\begin{note}
For more on approximations of this type, see \cite{Hed} and \cite{G}.
\end{note}

The proof of the theorem closely follows  that of \cite{for_wieg}, but with more explicit calculations. This eventually enables us to estimate the relation between the exactness of our approximation and the size of the domains on which the approximants are defined.

\begin{corollary} \label{cormerg}
	Let $\Omega \subset \C^n$ be a bounded domain with $C^0$-boundary. Assume that there exists a function $f: \mathbb{R}^+ \to \mathbb R$ and constants $c \geq 1$ and $\epsilon_1 < 1$ such that we have the following estimate on the distance function
	\begin{equation*}
		\delta(z,\partial \Omega) + f(\epsilon) \leq \delta(z+\epsilon w_j, \partial \Omega),
	\end{equation*}
	when $z \in B(x_j,r_j/c)$, $0 < \epsilon < \epsilon_1$ for some $j$, and $(x_j,r_j)$ are given by Definition \ref{defckgamma}, and $w_j \in \partial B(0,1)$ is the local direction given by the graph ($e_{2n}$ in local coordinates).
	Then there exists a constant $0 < \epsilon_0(\Omega,c,\epsilon_1) \leq \epsilon_1$ such that if $\phi \in \PSH(\overline \Omega) \cap C(\overline \Omega)$, and $0 < \epsilon < \epsilon_0$ are given, we can find a function $\psi$ and a domain $U_\psi \Supset \Omega$ such that $\psi\in\PSH(U_\psi)\cap C(U_\psi)$, and there exists a constant $C = C(\diam (\Omega),\Omega,c) > 1$ such that
	\begin{equation*}
		|\psi(z)-\phi(z)| \leq C\, \omega(\epsilon), \quad \forall z \in \overline \Omega,
	\end{equation*}
	where $\omega:\R^+ \to \R^+$ is the modulus of continuity of $\phi$ on $\overline \Omega$.
	Moreover we know that 
	\begin{equation*}
		f(\epsilon) < \delta(\Omega,\partial U_\psi).
	\end{equation*}
\end{corollary}
\begin{note}
	Since we know a lower bound on the distance between $\partial U_\psi$ and $\Omega$, we can mollify $\psi$ with a kernel of radius $f(\epsilon)/4$, hence obtaining a smooth approximation in $\hat U_\psi$, such that $d(\Omega, \partial \hat U_\psi) > f(\epsilon)/2$.
\end{note}
\subsection{Hyperconvexity in Log-Lipschitz domains and beyond}

The first result on the existence of bounded plurisubharmonic exhaustion functions is due to Diederich and Forna\ae{}ss. In 1977, \cite{DF}, they showed that any bounded pseudoconvex domain $\Omega \in \C^n$ with $C^r$ boundary, $r \geq 2$ has a $C^r$ defining function $\rho$ for which $-(-\rho)^\eta$ is a strictly plurisubharmonic bounded exhaustion function on $\Omega$ for any $0 < \eta$, small enough. They also asked if one can find a bounded plurisubharmonic exhaustion in the case of pseudoconvex $C^1$ domains. This was answered in 1981, when Kerzman-Rosay, \cite{KR}, proved that pseudoconvex $C^1$ domains are hyperconvex. This result was later generalised to Lipschitz domains by Demailly, \cite{Dem}, in 1987. For other results related to this question, see \cite{Ra}, \cite{Har} and \cite{SO}.

Elaborating on Demailly's proof, and keeping track of the dependence of the various estimates, we are able to extend Demailly's result to domains beyond Lipschitz regularity:

\begin{theorem} \label{thmhypconvex}
	Let $\Omega \subset \C^n$ be a bounded domain with $C^0$-boundary, and furthermore assume that there exists a $c \geq 1$ and $\epsilon_1 < 1$ such that we have the following estimate on the distance function
	\begin{equation*}
		\delta(z,\partial \Omega) + f(\epsilon) \leq \delta(z+\epsilon w_j, \partial \Omega)\leq \delta (z,\partial \Omega) + \epsilon,
	\end{equation*}
	when $z \in B(x_j,r_j/c)$, $0 < \epsilon < \epsilon_1$ for some $j$, where $(x_j,r_j)$ are given by Definition \ref{defckgamma}, and $w_j \in \partial B(0,1)$ is the local direction given by the graph ($e_{2n}$ in local coordinates).
	Where $f:\R_+ \to \R_+$ is a continuous function which satisfies the following limits
	\begin{eqnarray}
			\frac{\log \frac{\epsilon}{f(\epsilon)}}{\log \frac{1}{\epsilon}} &=& \omega(\epsilon)>0, \quad \text{for $0 < \epsilon < \epsilon_1$,} \label{eqfassump} \\
			\lim_{\epsilon \to 0^+} \omega(\epsilon) &=& 0. \notag
	\end{eqnarray}
	Then if $\Omega$ is pseudoconvex, it is also hyperconvex.
	Moreover there exists an $0 < \epsilon_0(\Omega,c,\epsilon_1) \leq \epsilon_1$, and a bounded plurisubharmonic exhaustion function $w \in C(\overline \Omega)$, which satisfies the following bound
	\begin{equation}
		\label{eqexhaustbddlow}
		-\frac{ \log(2)}{\log(1/\delta(z))} - C_1 \omega(\delta(z)) \leq w(z), \quad \text{for $\delta(z) \leq \epsilon_0$},
	\end{equation}
	for a constant $C_1(\Omega,c,\epsilon_1)>0$. If we in addition assume that $f$ is an increasing function we obtain the following upper bound.
	\begin{equation}
		\label{eqexhaustbddhigh}
		w(z) \leq \frac{\log \frac{f(\epsilon_0)}{\delta(z) + f(\epsilon_0)} }{\log 1/\epsilon_0}, \quad \text{for $\delta(z) < \epsilon_0$}.
	\end{equation}
\end{theorem}

If we additionally assume that $\Omega$ is a Log-Lipschitz domain, we also obtain estimates on the Levi form of the exhaustion function.

\begin{corollary} \label{corhypconvex}
	If $\Omega$ is a pseudoconvex Log-Lipschitz domain, then it is hyperconvex. Moreover there exists a constant $0 < \epsilon_0(\Omega,\|\Omega\|_{LL}) < 1$, and a smooth bounded plurisubharmonic exhaustion function $w$ satisfying \eqref{eqexhaustbddlow}, \eqref{eqexhaustbddhigh}, such that
	\begin{equation*}
		\ddc w(z) \geq C \frac{\log \big [\tilde C_1 \log \frac{1}{\delta(z)} \big]}{\log \frac{1}{\delta(z)} } \ddc |z|^2,
	\end{equation*}
	with a constant $C(\Omega,\|\Omega\|_{LL}) > 1$, for $\delta(z) < \epsilon_0$.
\end{corollary}

\section{A primer on boundary regularity} 
\label{sec:domains_with_low_regularity_rewrite_}

In complex analysis of several variables, the usual way of describing the regularity of a domain in $\Cn$, is using a defining function. To be precise
\begin{definition}
	A \emph{domain} is a connected open set.
\end{definition}
\begin{definition} \label{defck}
	A domain $\Omega \subset \Cn$ with boundary $\partial \Omega$ is said to have $C^k$ boundary, $k=1,2,\ldots$, if there is a $k$ times continuously differentiable function $\rho$ on a neighbourhood $U$ of $\partial \Omega$ such that
	\begin{enumerate}
		\item $\Omega \cap U = \{z \in U: \rho(z) < 0 \}$,
		\item $\nabla \rho \neq 0$ on $\partial \Omega$.
	\end{enumerate}
\end{definition}
This definition always provides you with a domain whose boundary is a $C^k$ manifold, however when $k < 1$, it no longer makes sense to talk about a defining function, since we cannot make a requirement on its gradient anymore, as such we need to look away from the above definition. We could consider the following common definition
\begin{definition} \label{defckgammaclass}
	A bounded domain $\Omega \subset \R^n$ is said to be of class $C^{k,\gamma}$, $k = 0,1,2, \ldots$, $\gamma \in (0,1]$, if for each point $p \in \partial \Omega$, there exists a radius $r > 0$ and a map $\phi_{p}: B(p,r) \to \R^n$ such that
	\begin{enumerate}
		\item $\phi_p$ is a bijection,
		\item $\phi_p$ and $\phi^{-1}_p$ are $C^{k,\gamma}$ functions, 
		\item $\phi_p(\partial \Omega \cap B(p,r)) \subset \{x \in \R^n | x_n = 0\}$,
		\item $\phi_p(\Omega \cap B(p,r)) \subset \{x \in \R^n | x_n > 0 \}$.
	\end{enumerate}
\end{definition}

Although this definition allows us to 'go below' $k = 1$, it has some serious drawbacks. When $k = 0$ and $\gamma < 1$, a domain of class $C^{k,\gamma}$, can exhibit unwanted behaviour, as the example 'Rooms and Passages' by Fraenkel shows, \cite[p.389]{fraenkel}. In the 'Rooms and Passages' example, the Sobolev embedding fails to hold, but also the segment property:

\begin{definition} \label{def:segmentpr}
 A bounded domain $\Omega$ in $\mathbb{R}^n$ has the \emph{segment property} if there, for every $z \in \partial \Omega$, exists a neighbourhood $U$ of $z$ and a vector $w \in \C^n$ such that $$U\cap \bar \Omega +tw \subset \Omega, \quad \forall \ 0<t<1.$$
 \end{definition}

The Hartogs triangle is a standard counterexample in complex analysis in several variables. It is easy to that it does not have the segment property.
\begin{example} \label{ex:hartogs}
Hartogs triangle $T=\{(z_1,z_2)\in \C^2 : |z_1|<|z_2|<1\}$ does not have the segment property. If it did, there would exist a neighbourhood $U$ of $0$ and a vector $w \in \C^2$ such that $z+tw \in T$ for every $z \in \bar T \cap U$ and $0<t<1$. This would mean that $tw \in T$ for given $0<t<1$. Since $T$ is a Reinhardt domain we know that $-tw \in T$. By choosing $t$ small enough, we could assume that $-tw \in T \cap U$. The segment property gives us that $-tw+tw \in T$, but we really have that $-tw+tw=0 \in \partial T$. Hence $T$ can not have the segment property.
\end{example}
In the theory of elliptic/parabolic partial differential equations in $\R^n$,
the embedding theorems are of vital importance, and hence the domains studied usually cannot be given in terms Definition \ref{defckgammaclass}. The following definition is commonly used instead:

\begin{definition} \label{defckgamma}
Let $k \in \{0,1,2,...\}$ and $\gamma \in [0,1]$. A bounded open set $\Omega \subset \R^n$ has \emph{$C^{k,\gamma}$-boundary} if there exists a finite set of balls $\{B(x_j,r_j)\}$, with $x_j\in
\partial\Omega$ and $r_j>0$, such that $\{B(x_j,r_j)\}$ constitutes a covering of an open neighbourhood of $\partial\Omega$ and such that, for each $j$, 
\begin{equation}
	\begin{split}
		\Omega\cap B(x_j, 4 r_j)=\{y=(y',y_n)\in\R^n : y_n > \phi_j ( y')\}\cap B(x_j, 4 r_j), \\
		\partial\Omega\cap B(x_j, 4 r_j)=\{y=(y',y_n)\in\R^n : y_n= \phi_j ( y')\}\cap B(x_j, 4 r_j), 
	\end{split}
	\label{eqn:lipschitz:local:description} 
\end{equation}
in an appropriate coordinate system, and the functions $\phi_j: \R^{n-1} \to \R$ lies in the class $C^{k,\gamma}$. The $C^{k,\gamma}$ norm of $\Omega$ is defined to be $M=\max_j\|\phi_j\|_{C^{k,\gamma}}$. If $k = 0$, we use the notation $C^{k,\gamma} = C^\gamma$, where $\gamma = 0$ means that we just have a continuous graph function.
\end{definition}

Definition \ref{defckgamma} excludes the 'Rooms and Passages' example, secondly Sobolev embedding holds for any value of $k,\gamma$, thirdly these domains satisfy the segment property:

\begin{lemma} \label{lem:segment}
A domain $\Omega \subset \mathbb{R}^n$ has the segment property if, and only if, $\Omega$ has $C^0$-boundary in the sense of Definition \ref{defckgamma}.
\end{lemma}

\begin{proof}
	For a proof see see \cite[Theorem 7.3]{DZ} or \cite[Theorem 3.3]{fraenkel}. The idea of the proof is to, first of all, point out that if the boundary is locally the graph of a continuous function, then there has to be a direction into the domain, hence the domain has the segment property. The proof of the converse is a bit more work. The main idea is o rotate and translate the domain then define the continuous function that locally defines the boundary, as the least distance from some plane to the boundary. Since the domain has the segment property, this functions will be continuous.
\end{proof}

 It should be noted that Definitions \ref{defck}, \ref{defckgammaclass} and  \ref{defckgamma} coincides, if $k \geq 1$, and $\gamma = 0$ using the identification $\R^{2n} \simeq \Cn$.

We will in the rest of the paper disregard Definitions \ref{defck} and \ref{defckgammaclass}, and refer only to Definition \ref{defckgamma}.

Another useful property of Definition \ref{defckgamma}, is that we can make precise estimates on the distance function. Making precise estimates on $\delta(z,\partial \Omega)$, is for example useful when working in a pseudoconvex domain $\Omega \subset \C^n$, we can then estimate the 'blow up' rate of the plurisubharmonic function $-\log \delta(z,\partial \Omega)$ as we approach the boundary. Such 'blow up' estimates will prove very useful later in this paper when we consider the question of pseudoconvexity + some boundary regularity $\implies$ hyperconvexity.

We will now provide some examples, on how to obtain these claimed estimates on the distance function, in some particular cases.
\begin{example}
	Let $\Omega \subset \R^n$ be a domain with $C^{0,1}$ boundary, usually denoted as, \emph{Lipschitz domains}. Without loss of generality, we may assume that our coordinate system is the standard one, $\partial \Omega \cap B(0,4)$ is given by $\phi:\R^{n-1} \to \R$, where $\phi$ is a Lipschitz function, with norm $C$.
	Then 
	\begin{equation} \label{eqlipest}
		\phi(x') \leq C|x'|, \quad x' \in B'(0,4).
	\end{equation}
	We would now like to calculate $\delta(\epsilon e_n, \partial \Omega)$, where $\epsilon e_n \in \Omega \cap B(0,1)$, $0 < \epsilon < 1$. Since our graph satisfies estimate \eqref{eqlipest},
	\begin{equation*}
		\delta(\epsilon e_n, \partial \Omega \cap B(0,1)) \geq \delta(\epsilon e_n, C|x|).
	\end{equation*}
	Notice that we have a strict relation between the distance from $\epsilon e_n$ to the cone $C|x|$, and the horizontal (along the hyperplane $\{x_n = 0\}$) distance to $C|x|$. Hence if $t$ solves
	\begin{equation*}
		C\, t = \epsilon,
	\end{equation*}
	then $t \leq C \delta(\epsilon e_n, C|x|)$.
	We have
	\begin{equation*}
		\epsilon/C \leq \delta(\epsilon e_n, \partial \Omega \cap B(0,1)).
	\end{equation*}
	We also have the obvious upper bound
	\begin{equation*}
		\delta(\epsilon e_n, \partial \Omega \cap B(0,1)) \leq \epsilon.
	\end{equation*}
	Returning to the general setting we get
	\begin{equation*}
		\delta(x)+\epsilon/C \leq \delta(x+\epsilon e_n, \partial \Omega) \leq \delta(x) + \epsilon,
	\end{equation*}
	for $x \in B(0,1) \cap \Omega$.
\end{example}
\begin{example}
	Let $\Omega \subset \R^n$ be a domain with $C^{0,\gamma}$ boundary, with $\gamma \in (0,1)$, usually denoted as, \emph{H\"older domains}. Proceeding as in the above example, we obtain an estimate of the form
	\begin{equation*}
		\delta(z)+(\epsilon/C)^{1/\gamma} \leq \delta(z+\epsilon e_n, \partial \Omega) \leq \delta(z) + \epsilon.
	\end{equation*}
\end{example}

We will now provide a class of domains which lies in between the Lipschitz domains and H\"older domains, in terms of regularity.
First let us define the space of functions our graph functions will lie in:
\begin{definition} \label{defloglip1}
	Let $f: \Omega \to \R$, we say that $f$ is \emph{Log-Lipschitz continuous} if for each ball $B(x,r) \subset \Omega$, for $r < 1/10$, we have for each $x,y \in B(x,r)$
	\begin{equation*}
		|f(x) - f(y)| \leq C |x-y|\log (1/|x-y|),
	\end{equation*}
	with $C(B(x,r)) > 0$. The number $||f||_{LL}:=C$ is called the \emph{Log-Lipschitz norm} of the function $f$.
\end{definition}

Let us give an example of a Log-Lipschitz continuous function which is not just the obvious function $|x|\log(1/|x|)$.

\begin{example} \label{excusp}
	Let $f(x) = x e^x$, for $x \in \C$, then the multivalued function $W(x) = f^{-1}(x)$, is called the Lambert $W$ function. For real valued $x$ we obtain two branches for $-1/e < x < 0$. Let $W_{-1}:(-1/e,0) \to (-\infty,-1]$ be the lower branch, and denote $W_0:(-1/e,\infty) \to [-1,\infty)$, the upper branch. Then the function 
	\begin{equation*}
		f(x) = 1 + |x| W_0(1/|x|),
	\end{equation*}
	is the cusp shown in Figure \ref{fig1}. Which satisfies Definition \ref{defloglip1}, and hence is a Log-Lipschitz function. Moreover the distance between two vertical translates is given by $-\epsilon/\log(\epsilon)$.
\end{example}

\begin{figure}[!ht]
	\begin{center}
		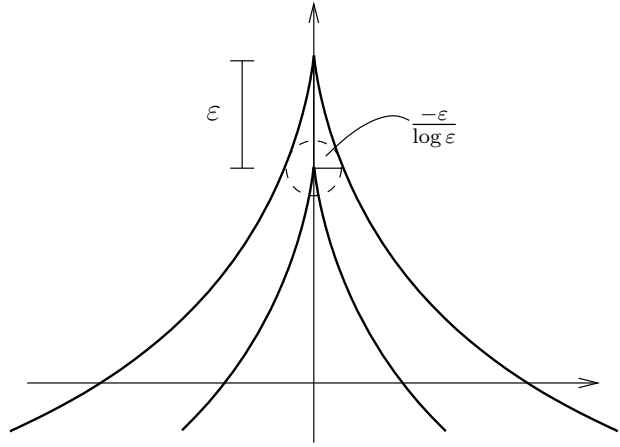
	\end{center}
	\caption{An example cusp, which is H\"older for each $0 < \alpha < 1$ but not Lipschitz}
	\label{fig1}
\end{figure}

We are now in a position to define Log-Lipschitz domains.

\begin{definition} \label{defloglip}
	A bounded open set $\Omega \subset \R^n$ has \emph{Log-Lipschitz-boundary} if there exists a finite set of balls $\{B(x_j,r_j)\}$, with $x_j\in
	\partial\Omega$ and $r_j>0$, such that $\{B(x_j,r_j)\}$ constitutes a covering of an open neighbourhood of $\partial\Omega$ and such that, for each $j$, 
	\begin{equation*}
		\begin{split}
			\Omega\cap B(x_j, 4 r_j)=\{y=(y',y_n)\in\R^n : y_n > \phi_j ( y')\}\cap B(x_j, 4 r_j), \\
			\partial\Omega\cap B(x_j, 4 r_j)=\{y=(y',y_n)\in\R^n : y_n= \phi_j ( y')\}\cap B(x_j, 4 r_j), 
		\end{split}
	\end{equation*}
	in an appropriate coordinate system, and the functions $\phi_j: \R^{n-1} \to \R$ is Log-Lipschitz continuous. The Log-Lipschitz norm of $\Omega$ is defined to be $M=\max_j\|\phi_j\|_{LL}$.
\end{definition}

Again, control on the regularity of the graph function, gives control on the distance function:

\begin{lemma} \label{lemloglip}
	Let $\Omega \subset \R^n$ be a domain with Log-Lipschitz boundary. Let $x \in \Omega \cap B(x_j,r_j)$, with $(x_j,r_j)$ given by Definition \ref{defloglip}, then in the local coordinates associated to $(x_j,r_j)$, we have
	\begin{equation} \label{eqdist}
		\delta(x) + \frac{-\epsilon}{\tilde C W_{-1}(-\epsilon / C)} \leq \delta(x+\epsilon e_n) \leq \delta(x) + \epsilon,
	\end{equation}
	for any $0 < \epsilon < \epsilon_0(r_j)$, and $C>0$ and $\tilde C>0$, depends on the Log-Lipschitz constant of the domain. $W$ is the Lambert-W function defined in Example \ref{excusp}.
\end{lemma}

\begin{proof}
	Without loss of generality we may assume that $x_j = 0$ and $r_j = 1/10$, and $\phi_j$ is Log-Lipschitz continuous, with norm $C$, and $\phi_j(0) = 0$. Then
	\begin{equation*}
		\phi_j(x') \leq C |x'| \log \frac{1}{|x'|} \quad \forall x' \in B'(0,1/10).
	\end{equation*}
	We would like to calculate $\delta(\epsilon e_n, \partial \Omega)$, where $\epsilon e_n \in \Omega \cap B(0,1/10)$ if $0 < \epsilon < 1/10$.
	\begin{equation*}
		\delta(\epsilon e_n, \partial \Omega) \geq \delta(\epsilon e_n, C|x'|\log \frac{1}{|x'|}), \quad x' \in B'(0,1/10).
	\end{equation*}
	Where we obtain that if $t > 0$ solves the following equation
	\begin{equation*}
		C t \log \frac{1}{t} = \epsilon,
	\end{equation*}
	then $t \leq \hat C \delta(\epsilon e_n,C|x'|\log \frac{1}{|x'|})$ with $x' \in B'(0,1/10)$.
	The solution $t$ is given as
	\begin{equation*}
		t=\frac{-\epsilon}{C W_{-1}(-\epsilon / C)}.
	\end{equation*}
	Hence we obtain the estimate
	\begin{equation*}
		\frac{-\epsilon}{\tilde C W_{-1}(-\epsilon / C)} \leq \delta(\epsilon e_n, \partial \Omega),
	\end{equation*}
	which obviously implies estimate \eqref{eqdist}.
\end{proof}
\section{Proof of Theorem \ref{theorem:c0bdry}}\label{sec:pshmerg}
	For ease of notation, assume that $0 \in \Omega$.	Since $\Omega$ has a $C^0$ boundary, according to Definition \ref{defckgamma} there exists finite collection of balls $W_i = B(x_j,4r_j)$, $j=1,\ldots, m$, such that $\{B(x_j,r_j)\}$ covers $\partial \Omega$. Together with a set of unit vectors $w_1,\ldots,w_m$ in $\C^n$ representing the $e_{2n}$ direction in the local coordinates associated to the graph function.
	Let $\epsilon_w = \min_{j} r_j$. Construct $W_0 = \{z \in \Omega: \delta(z,\partial \Omega) > \epsilon_w \}$, then $\{W_j\}_{j=0}^m$ forms an open cover of $\overline \Omega$. Moreover $\big \{ \{ z \in W_j: \delta(z,\partial W_j) > \epsilon_w \} \big \}_{j=0}^m$ is still an open cover of $\overline \Omega$. Set
	\begin{equation*}
		B_j = \{ z \in W_j: \delta(z,\partial W_j) > \epsilon_w/2 \}, \quad \text{and} \quad B^-_j = \{ z \in W_j: \delta(z,\partial W_j) > \epsilon_w \}.
	\end{equation*}
	Let now 
	\begin{equation*}
		d_j = \delta\bigg (\partial B_j \cap \overline \Omega,\partial \bigg (\bigcup_{k \neq j} B^-_k \bigg ) \bigg ) \geq \epsilon_w/2 > 0,
	\end{equation*}
	where the inequality comes from the fact that $\cup B^-_k$ is an open cover of $\overline \Omega$.
	We can now construct a family of compact sets, let
	\begin{equation*}
		K_{j,k} = \{z \in \overline \Omega: \delta(z,\partial B_j) \leq d_j \} \cap \overline{B^-_k}.
	\end{equation*}
	We can now easily see two things, first is that $K_{j,k}$ is compact and the second that $K_k = \bigcup_{j} K_{j,k} \subset B_k$. We can now see that $\delta(K_k,\partial B_k) \geq {\epsilon_w}/{2}$, hence we can form the cutoff-functions $\xi_k$ as follows, let $-1 \leq \xi_k \leq 0$, $\xi_k(z) = 0$ if $\delta(z,K_k) \leq {\epsilon_w}/{4}$ and $\xi_k \equiv -1$ outside of $B_k$. Notice that $||{\mathcal{L} \xi_k}|| \leq C /(\epsilon_w)^2$ hence $\xi_k + c|z|^2$ where $c=c(\epsilon_w) \approx 1/(\epsilon_w)^2$, is plurisubharmonic in $\C^n$.
	
	Let $\varphi \in \PSH(\Omega)\cap C(\bar \Omega)$. By letting
	\begin{equation*}
		\varphi_j(z)=\varphi(z+ \nu w_j),
	\end{equation*}
	with $0 < \nu < \epsilon_w$ and  $w_0 := 0$, we see that for all $j$, it holds that
	\begin{equation} \label{eq:epsilon_j:1}
	|\varphi(z) - \varphi_j(z)|< \omega(\nu) \ \text{on} \ \bar \Omega \cap \bar B_j.
	\end{equation}
	This is possible since $\varphi$ is uniformly continuous on $\bar \Omega$, with modulus of continuity given by the function $\omega:\R_+ \to \R_+$. Also note that $\varphi_j$ is plurisubharmonic and continuous in $U_j \equiv \Omega \cap W_j - \nu w_j \supset \overline \Omega \cap \overline B_j$, if $\nu < \epsilon_w/2$. We will now use $\varphi_j$ to construct a family of functions, whose maximum will approximate $\varphi$. Let for $z \in U_j  \cap \overline B_j$
	\begin{equation*}
		f_j(z) = \varphi_j(z) + 3 \omega(\nu) (\xi_j(z) + c|z|^2),
	\end{equation*}
	elsewhere define the function $f_j = -\infty$. The maximum of the family
	\begin{equation*}
		v(z) = \max f_j(z),
	\end{equation*}
	will be our candidate approximating function.
	To show that $v$ approximates $\varphi$ uniformly on $\overline \Omega$, we observe that for $z \in \overline \Omega$ we have $v(z) = f_{j_z}(z)$, for some $j_z \in \{1,\ldots,m\}$, moreover according to \eqref{eq:epsilon_j:1}
	\begin{eqnarray*}
		|\varphi(z) - v(z)| &=& \bigg |\varphi(z) - \bigg (\varphi_{j_z}(z) + 3 \omega(\nu) (\xi_{j_z}(z) + c|z|^2) \bigg ) \bigg| \\
		&\leq& \omega(\nu)\big [1+C\, \diam (\Omega) \big ].
	\end{eqnarray*}
	Note that the above estimate gives us uniform approximation, since we can choose $\nu$ as small as we want to obtain an approximation arbitrarily close.
	
	Next we will show that $v$ is in fact plurisubharmonic and continuous on a neighbourhood of $\overline \Omega$, namely $U = \bigcup_j (U_j \cap \overline B_j)$. First take a $z \in U$ that is not on the boundary of any $\partial B_j$, then the maximum is taken of a finite amount of continuous and plurisubharmonic functions in a neighbourhood of $z$. If $z \in \partial B_j \cap U$, then there exists a $k$ such that $\delta(z,K_k) \leq \epsilon_w/2 $, for this $j$ and $k$, we have that
	\begin{eqnarray*}
		f_j(z) &=& \varphi_j(z) + 3 \omega(\nu) (\xi_j(z) + c|z|^2) \\
		&=& \varphi_j(z) + 3 \omega(\nu) (c|z|^2 - 1) \\
		&=& \bigg ( \varphi_j(z) - \varphi_k(z) \bigg ) + \bigg (\varphi_k(z) + 3 \omega(\nu)(\xi_k(z) + c |z|^2) \bigg ) - 3 \omega(\nu) \\
		&\leq& - \omega(\nu)+f_k(z) < f_k(z),
	\end{eqnarray*}
	since $\xi_k(z) = 0$. This estimate means that locally near $z$, we can assume that the function $v$ is the maximum of functions $f_k$, $k \neq j$, that are continuous and plurisubharmonic in a neighbourhood of $z$.
	We have now proved that in $U$ the function $v$ is plurisubharmonic and continuous. To conclude, we mollify $v$ with a small enough kernel, guarantied since $\Omega \Subset U$.
	\begin{flushright}
		\qedsymbol
	\end{flushright}

\section{Proofs of Theorem \ref{thmhypconvex} and Corollary \ref{corhypconvex}} 

\subsection{Proof of Theorem \ref{thmhypconvex}}

Fix $\epsilon > 0$, again since $\Omega$ is $C^0$, there exists a finite number of points $x_j \in \partial \Omega$, radii $r_j > 0$, and unit vectors $w_j \in \C^n$, such that for all $j$:
\begin{equation} \label{eqlip}
	(\Omega \cap B(x_j,r_j))+ t w_j \subset \Omega \quad \forall 0 < t < \epsilon_w = \min_j r_j.
\end{equation}
For $z \in \overline B(x_j,r_j/2)$ and $0 < \epsilon < \epsilon_w/c$, then
\begin{equation} \label{eqlipbd}
	\delta(z)+f(\epsilon) \leq \delta(z+\epsilon w_j) \leq \delta(z) + \epsilon,
\end{equation}

Let
\begin{equation}
	\label{eqvepsilonj}
	v_{\epsilon,j} = \log \frac{1}{\delta(z+\epsilon w_j)}, \quad
	z \in \Omega \cap \overline B(x_j,r_j/2),
\end{equation}
then $v_{\epsilon,j}$ is plurisubharmonic since $\log \frac{1}{\delta(z)}$ is, and \eqref{eqlipbd} gives
\begin{equation*}
	\log \frac{1}{\delta(z)+\epsilon} \leq v_{\epsilon,j}(z) \leq \log \frac{1}{\delta(z) + f(\epsilon)}.
\end{equation*}
It remains to combine the functions $v_{\epsilon,j}$ in such a way that we obtain a global plurisubharmonic function $v_\epsilon$ on $\Omega$.

Let $\psi_j \in C_0^\infty(B(x_j,r_j/2))$ , such that $\psi_j = \log(\epsilon/f(\epsilon))$ on $B(x_j,r_j/3)$. 
Note that $\ddc \psi_j \leq C \log(\epsilon/f(\epsilon))\ddc \abs{z}^2$.
Let $\lambda = C \log (\epsilon/f(\epsilon))$, where $C > 1$ is a constant, then $\psi_j(z) + \lambda(\epsilon) |z|^2$ is plurisubharmonic.
Moreover we obtain the upper and lower bounds for each $j$.
\begin{equation*}
	v_{\epsilon,j}(z) + \psi_j(z) + \lambda |z|^2 > \lambda |z|^2 + \log \frac{1}{\delta(z)+f(\epsilon)},
\end{equation*}
on $B(x_j,r_j/3)$, and
\begin{equation*}
	v_{\epsilon,j}(z) + \psi_j(z) + \lambda |z|^2 \leq \lambda |z|^2 + \log \frac{1}{\delta(z)+f(\epsilon)},
\end{equation*}
on $\partial B(x_j,r_j/2)$.
Now let for each $z \in \Omega$ define
\begin{equation}
	\label{eqvepsilon}
	v_{\epsilon} (z) = \max_{j} \bigg (v_{\epsilon,j}(z) + \psi_j(z) + \lambda |z^2|-\gamma \lambda, |z|^2 - \lambda \bigg),
\end{equation}
the maximum is taken over all indices $j$ such that $z \in \overline B(x_j,r_j/2)$, if no such index exists, we let $v_{\epsilon,j} = |z|^2 - \lambda$. For $\gamma > 1$ sufficiently large, the function $v_\epsilon(z)$ is plurisubharmonic and continuous on $\Omega$.
Also note that as $\epsilon \to 0$ the above function will be $+\infty$ on $\partial \Omega$. 

Now we construct the function $\phi$ with help of the functions $v_\epsilon$. Observe that \eqref{eqfassump}, gives that $\lim_{\epsilon \to 0^+} \frac{\lambda}{\log \frac{1}{\epsilon}} = 0$.
That is,
\begin{equation}
	\label{eqbdd}
	\log \frac{1}{\delta(z)+\epsilon} - \hat C_1 \lambda\leq v_\epsilon (z) \leq \log \frac{1}{\delta(z)+f(\epsilon)} - \hat C_2 \lambda,
\end{equation}
for constants $\hat C_1, \hat C_2 > 1$.
To continue let us set
\begin{equation*}
	w_\epsilon (z) = \frac{v_\epsilon (z)}{\log(1/\epsilon)} - 1,
\end{equation*}
and
\begin{equation}\label{eq:w_sup}
	w(z) = \sup_{0 < \epsilon \leq \epsilon_0} w_\epsilon(z),
\end{equation}
where $\epsilon_0 = e_w/c$. 

Next we will build estimates for the family of function $w_\epsilon$ in order to obtain estimates for $w$.
Using \eqref{eqbdd} we obtain the following
\begin{equation}
	\frac{\log \frac{1}{\delta(z)+\epsilon} - \hat C_1 \lambda(\epsilon) - \log \frac{1}{\epsilon}}{\log(1/\epsilon)}
	\leq 
	w_\epsilon (z) 
	\leq 
	\frac{\log \frac{1}{\delta(z)+f(\epsilon)} - \hat C_2 \lambda(\epsilon) - \log \frac{1}{\epsilon}}{\log(1/\epsilon)}.
\end{equation}

First let us verify that $w_\epsilon < 0$ for all $0 < \epsilon < \epsilon_0$.
Note that this is equivalent to verifying
\begin{equation*}
	\log \frac{f(\epsilon)}{\delta(z) + f(\epsilon)} < 0,
\end{equation*}
which is obviously true.

The choice $\epsilon = \delta(z)$ gives the lower bound
\begin{equation} \label{eqlower}
	w(z) \geq w_\epsilon (z) \geq -\frac{ \log(2)}{\log(1/\delta(z))} - \hat C_1 \omega(\delta(z)).
\end{equation}

To conclude that $w$ is continuous, note that $w_\epsilon(z)$ is a continuous function in terms of $(\epsilon,z)$, since $f \in C^0$, fixing a small neighbourhood around $z$, gives us that in this neighbourhood the supremum is taken for $\epsilon \geq \hat \epsilon$, then we can via a compactness argument deduce that $w$ is continuous at $z$. Using $\eqref{eqlower}$ and the fact that $w < 0$, we see that $w (z) \to 0$ as $z \to \partial \Omega$, which implies that $w \in C^0(\bar \Omega)$. Since we now know that $w$ is continuous it is obviously plurisubharmonic since it is locally the supremum over a family of plurisubharmonic functions.

If we in addition assume that $f$ is an increasing function, notice that
\begin{equation*}
	\log \frac{f(\epsilon)}{\delta(z) + f(\epsilon)}
\end{equation*}
is an increasing function with respect to $\epsilon$, and as such we can obtain the improved upper bound

\begin{equation} \label{equpper}
	w_\epsilon (z) \leq \frac{\log \frac{1}{\delta(z)+f(\epsilon)} - \hat C_2 \lambda(\epsilon) - \log \frac{1}{\epsilon}}{\log(1/\epsilon)} \leq \frac{\log \frac{f(\epsilon_0)}{\delta(z) + f(\epsilon_0)} }{\log 1/\epsilon_0} < 0.
\end{equation}
The upper estimate \eqref{equpper} is actually a linear upper bound. Leaving a big gap between the lower bound \eqref{eqlower} and the upper bound \eqref{equpper}, however this is natural due to the gap between the upper and lower estimate on the distance function, together with the compensating factor $\lambda$ which depends on $\epsilon$.
\begin{flushright}
	\qedsymbol
\end{flushright}

\subsection{Proof of Corollary \ref{corhypconvex}}
First we need the following lemma:
	\begin{lemma}
		If we let $f(\epsilon) = \frac{-\epsilon}{\tilde C W_{-1}(-\epsilon / C)}$, then the following limit holds.
		\begin{equation} \label{eqlim}
			\lim_{\epsilon \to 0^+} \frac{\log \frac{\epsilon}{f(\epsilon)}}{\log \frac{1}{\epsilon}} = 0.
		\end{equation}
		Furthermore the function $f$ is increasing for $0 < \epsilon < C/e$.
	\end{lemma}
	\begin{proof}
		Observe that $W_{-1}(-x) \approx \log(x)$ as $x \to 0^+$, hence
		\begin{eqnarray*}
			\lim_{\epsilon \to 0^+} \frac{\log \frac{\epsilon}{f(\epsilon)}}{\log \frac{1}{\epsilon}} &=& \lim_{\epsilon \to 0^+} \frac{\log\left ( \tilde C W_{-1}(-\epsilon / C)  \right )}{\log \frac{1}{\epsilon}} \\
			&=& \lim_{\epsilon \to 0^+} \frac{\log\left ( -\tilde C \log (\epsilon / C)  \right )}{\log \frac{1}{\epsilon}} = 0.
		\end{eqnarray*}
		To prove that $f$ is increasing, note that
		\begin{equation*}
			\frac{\partial f}{\partial \epsilon} = -\frac{1}{\tilde C (1+W_{-1}(-\epsilon/C))} > 0
		\end{equation*}
		since $W_{-1}(-\epsilon/C) < -1$ if $\epsilon/C < 1/e$.
	\end{proof}
	Notice that due to the lemma above together with Lemma \ref{lemloglip}, we satisfy all the requirements of Theorem \ref{thmhypconvex}, and as such we obtain a bounded plurisubharmonic exhaustion function $w$ together with the bounds \eqref{eqexhaustbddlow} and \eqref{eqexhaustbddhigh}.
	What is left to do is to show that $\ddc w(z)$ is bounded from below, then we use this together with a regularisation by Richberg to obtain a smooth bounded plurisubharmonic exhaustion function satisfying the theorem.
	
	To bound $\ddc w(z)$, let us go back to the proof of Theorem \ref{thmhypconvex}, giving this estimate
	\begin{equation*}
		\frac{\log \frac{1}{\delta(z)+\epsilon} - \hat C_1 \lambda(\epsilon) - \log \frac{1}{\epsilon}}{\log(1/\epsilon)}
		\leq 
		w_\epsilon (z) 
		\leq 
		\frac{\log \frac{1}{\delta(z)+f(\epsilon)} - \hat C_2 \lambda(\epsilon) - \log \frac{1}{\epsilon}}{\log(1/\epsilon)}.
	\end{equation*}
	Then if we can show that for $\epsilon < \hat c \delta(z)$, for some small $\hat c > 0$, then we know that the supremum in \eqref{eq:w_sup} is attained for $\epsilon \geq \hat c \delta(z)$. Using the following bound on $\ddc w_\epsilon$
	\begin{equation*}
		\ddc w_\epsilon(z) \geq C \frac{\log \big [\tilde C_1 \log \frac{1}{\epsilon} \big]}{\log \frac{1}{\epsilon} } \ddc |z|^2,
	\end{equation*}
	together with the lower bound on $\epsilon$, we obtain
	\begin{equation} \label{eqddcbdd}
		\ddc w(z) \geq C \frac{\log \big [\tilde C_1 \log \frac{1}{\delta(z)} \big]}{\log \frac{1}{\delta(z)} } \ddc |z|^2.
	\end{equation}
	Notice that in the above we have used the fact that if $\ddc w_\epsilon(z) \geq \gamma$, for all $\epsilon \in [\epsilon_{\text{min}} , \epsilon_{\text{max}}]$, then $\ddc \sup_{\epsilon} w_\epsilon(z) \geq \gamma$.
	
	To show that the supremum is not attained for small $\epsilon$ it is enough to show that the following holds for small $\epsilon$,
	\begin{equation}
		\label{eqeps0}
		\frac{\log \frac{1}{\delta(z)+\epsilon_0} - \hat C_1 \lambda(\epsilon_0) - \log \frac{1}{\epsilon_0}}{\log(1/\epsilon_0)} \geq
		\frac{\log \frac{1}{\delta(z)+f(\epsilon)} - \hat C_2 \lambda(\epsilon) - \log \frac{1}{\epsilon}}{\log(1/\epsilon)}.
	\end{equation}
	The right hand side can be rewritten as
	\begin{equation*}
		\frac{\log \frac{f(\epsilon)}{\delta(z)+f(\epsilon)} - \hat C_3 \lambda(\epsilon)}{\log(1/\epsilon)} = 
		\frac{\log \frac{f(\epsilon)}{\delta(z)+f(\epsilon)}}{\log(1/\epsilon)} - \hat C_3 \omega(\epsilon) =: A(\epsilon) - B(\epsilon).
	\end{equation*}
	Notice that $A \to -\infty$ and $B \to 0$ as $\epsilon \to 0^+$, this gives that \eqref{eqeps0} is satisfied for small enough $\epsilon$, however we now need to estimate the size of $\epsilon$ from above. Since our graph is Log-Lipschitz, we can assume that
	$$f(\epsilon) = \frac{\epsilon}{\tilde C_1 \log(1/\epsilon)},$$
	since $-W_{-1}(-\epsilon/C) \geq \frac{1}{\hat C} \log(1/\epsilon)$, for some constant $\hat C(\epsilon_0) > 1$, $\epsilon < \epsilon_0$.
	We then obtain
	\begin{eqnarray*}
		A(\epsilon) = \log \frac{\epsilon}{(\tilde C_1 \log(1/\epsilon)) \delta(z) + \epsilon} \leq \log \frac{\epsilon}{(\tilde C_1 \log(1/\epsilon)) \delta(z)}.
	\end{eqnarray*}
	Assume now that $\delta(z) \leq \epsilon_0$ then it is enough to show that for small $\epsilon$ we have
	\begin{equation*}
		\eta := \frac{\log \frac{1}{2\epsilon_0} - \hat C_1 \lambda(\epsilon_0) - \log \frac{1}{\epsilon_0}}{\log(1/\epsilon_0)} \geq \log \frac{\epsilon}{(\tilde C_1 \log(1/\epsilon)) \delta(z)} - B(\epsilon).
	\end{equation*}
	To verify this we first find for which $0 < \epsilon << 1$ the right hand side is an increasing function with respect to $\epsilon$, then we find an $\epsilon$ in that range which satisfies the above inequality, and we are done.
	We want to show \eqref{eqeps0}, for $\epsilon < \hat c \delta(z)$, with $\hat c$ a small positive constant, to be chosen.
	To show this we first show that for small enough $\epsilon$ the following is true
	\begin{equation} \label{eqepderiv}
		\frac{\partial}{\partial \epsilon} \left (  \log \frac{\epsilon}{(\tilde C_1 \log(1/\epsilon)) \delta(z)} - B(\epsilon) \right ) > 0,
	\end{equation}
	which is implied if the following is true
	\begin{equation*}
		\hat C_3 + \log \frac{1}{\epsilon} + \left(\log \frac{1}{\epsilon}\right)^2- \hat C_3 \log \big (\tilde C_1 \log \frac{1}{\epsilon} \big ) > 0,
	\end{equation*}
	which is obviously true for $0 < \epsilon < \epsilon_1(\hat C_3, \tilde C_1)$ small enough. Secondly we show that
	\begin{equation} \label{eqgammabound}
		\eta \geq \log \frac{\epsilon}{(\tilde C_1 \log(1/\epsilon)) \delta(z)}
	\end{equation}
	for some $\epsilon < \epsilon_1$. To show \eqref{eqgammabound}, we first note that since we made the assumption that $\epsilon < \hat c \delta(z)$, we have
	\begin{equation*}
		\log \frac{\hat c \delta(z)}{(\tilde C_1 \log 1/(\hat c \delta(z)))\delta(z)} \geq \log \frac{\epsilon}{(\tilde C_1 \log(1/\epsilon)) \delta(z)},
	\end{equation*}
	we also note that there exists $\hat c(\epsilon_0) > 0$ such that
	\begin{equation*}
		\log \frac{\hat c }{\tilde C_1 \log \frac{1}{\hat c \epsilon_0}} \leq \eta.
	\end{equation*}
	We are now done since $\delta(z) < \epsilon_0$, and the above implies \eqref{eqgammabound}, which together with \eqref{eqepderiv} implies \eqref{eqeps0}, finally giving us \eqref{eqddcbdd}.

	We will now use the following regularisation, \cite[Thm 5.21]{Dembook}, originally due to Richberg (\cite{R}), reformulated for our use.
	\begin{lemma} \label{lemRich}
		Let $w \in \PSH(\Omega) \cap C^0(\Omega)$ which is strictly plurisubharmonic, with $\ddc w \geq \gamma$, for some positive (1,1)-form $\gamma$ on $\Omega$. For any continuous function $\lambda \in C^0(\Omega)$, $\lambda > 0$, there exists a plurisubharmonic function $\tilde w \in \C^\infty(\Omega)$, such that $w \leq \tilde w \leq w + \lambda$ on $\Omega$, which satisfies $\ddc \tilde w \geq (1-\lambda) \gamma$.
	\end{lemma}
	We will now use this lemma, where 
	\begin{eqnarray*}
		\gamma &=& C \frac{\log \big [\tilde C_1 \log \frac{1}{\delta(z)} \big]}{\log \frac{1}{\delta(z)} } \ddc |z|^2, \\
		\lambda &=& - \frac{w}{2}.
	\end{eqnarray*}
	Hence by Lemma \ref{lemRich}, there exists $\tilde w$ satisfying the following bounds
	\begin{equation*}
			-\frac{ \log(2)}{\log(1/\delta(z))} - C_1 \omega(\delta(z)) \leq \tilde w(z)
			\leq \frac{1}{2}\frac{\log \frac{f(\epsilon_0)}{\delta(z) + f(\epsilon_0)} }{\log 1/\epsilon_0}, \quad \delta(z) \leq \epsilon_0
	\end{equation*}
	To be precise we will extend $\tilde w$ to $\partial \Omega$, as $0$, hence obtaining $\tilde w \in C^0(\bar \Omega)$.
	Moreover
	\begin{eqnarray*}
		\ddc \tilde w(z) &\geq& (1+w(z)/2) C \frac{\log \big [\tilde C_1 \log \frac{1}{\delta(z)} \big]}{\log \frac{1}{\delta(z)} } \ddc |z|^2 \\
		&\geq&
		\left ( 1-\frac{ \log(2)}{\log(1/\delta(z))} - C_1 \omega(\delta(z))  \right )C \frac{\log \big [\tilde C_1 \log \frac{1}{\delta(z)} \big]}{\log \frac{1}{\delta(z)} } \ddc |z|^2
	\end{eqnarray*}
	for $\delta(z) < \epsilon_0$, redefining $\epsilon_0 > 0$ to be small enough we obtain our result, with a new constant $C > 0$.
\begin{flushright}
	\qedsymbol
\end{flushright}

\subsection*{Acknowledgements}
We would like to thank U. Cegrell and P. \AA{}hag of Ume\aa{} University for inspiring discussions on the theme of this paper.

\end{document}